\documentclass[10 pt]{amsart}
\usepackage{latexsym, amsmath, amssymb, longtable, booktabs,amscd,microtype,booktabs,cases}
\usepackage[square]{natbib}
\usepackage[english]{babel}
\usepackage[utf8x]{inputenc}
\usepackage{hyperref}
\bibpunct{[}{]}{,}{n}{}{;}

\usepackage[english]{babel}
\usepackage[utf8x]{inputenc}
\usepackage{graphicx}
\usepackage{cleveref}

\usepackage[OT2,T1]{fontenc}
\DeclareSymbolFont{cyrletters}{OT2}{wncyr}{m}{n}
\DeclareMathSymbol{\Sha}{\mathalpha}{cyrletters}{"58}
\numberwithin{equation}{section}

\newtheorem*{deff}{Definition}
\newtheorem*{thh}{Proposition}
\newtheorem*{thhno}{Theorem}

\usepackage[switch,columnwise]{lineno}
\usepackage{lipsum} 

\usepackage{fancyhdr,lipsum}
\date{\today}

\begin{document}

\author{HOOSHANG SAEID--NIA}
\address{HOOSHANG SAEID--NIA
\ \ }
\email{hooshang.s.nia@gmail.com   - - -   Social Page: www.instagram.com/h.s.nia}

\title{
There Are No Odd Perfect Numbers\\
}

\smallskip
In Memory of Maryam Mirzakhani (1977-2017)

\begin{abstract}
While the general form of even perfect numbers is well-known, the exist\-ence or non-exist\-ence  
of odd perfect numbers is still an open problem. We address this problem and prove that if a natural number is odd,
then it's not perfect.
\end{abstract}

\maketitle
\section{introduction}
The following is one of the ancient open problems in number theory, perhaps \cite{r1} the oldest open problem in all mathematics,
at the time of writing:

\begin{deff}(Perfect Number)
A natural number $n$ is said to be perfect if the
sum of all its [positive] divisors, including $n$ itself, is equal to $2n$.
\begin{equation}
\sum_{d| n}d=2n \mbox{, equivalently: } \sum_{\substack{d| n \\ d < n}}d=n.
\label{M}
\end{equation}
\end{deff}

{\bf Example:} 6, 28, 496, and 8128 are the first few perfect numbers. $\lhd$

All even perfect numbers are completely determined \cite{r1} by the following theorem:

\begin{thhno}{Even Perfect Numbers}\\
{\bf A)} Euclid (300 B.C.) if $2^p-1$ is prime, then $n = 2^{p-1}(2^p-1)$ is a perfect number. (Elements, Book IX, Proposition 36, as cited in \cite{r1}). \\
{\bf B)} Euler (1707 -- 1783) If $n$ is even, then the converse of part (A) is also true. \cite{r1} i.e. even perfect numbers 
{\it *must*} be of the form given by Euclid in part (A).
\end{thhno}

The ancient open problem is whether or not any {\it odd} perfect numbers exist? We answer that question, 
as already suggested by many authors \cite{r3}, negatively; using only elementary tools.

A good account of previous work on this topic can be found, for example in \cite{r2} and \cite{r3}.

\section{Preliminaries}

It's understood that 1 is not a perfect number. The sum of divisors of 1 is 1, which is not two times 1. Therefore, when we look for odd perfect numbers, $n > 1$, and it can be uniquely factorized as

\begin{displaymath}
n=\prod_{i=1}^{m} p_i^{a_i}   \mbox{\quad (} p_i \mbox{ odd prime)}
 \mbox{ (} a_i \mbox{ positive).}
\end{displaymath}

The sum of divisors of $x$ - usually denoted by $\sigma(x)$ - for a prime power is:

\begin{displaymath}
\sigma(p^a) =\sum_{j=0}^{a}p^{j},
\end{displaymath}

and because $\sigma$ is a multiplicative function (\cite{r1}), for every $n$ in general, it can be written as:

\begin{equation}
\sigma(n) = \prod_{i=1}^{m} (\sum_{j=0}^{a_i}p_i^{j}),
\label{sigm}
\end{equation}

{\bf Example:} For $n = 3^4$ we have,

\begin{displaymath}
\sigma(n) = \sigma(3^4) = \sum_{j=0}^{4}3^{j} = 1 + 3 + 3^2 + 3^3 + 3^4 = 121,
\end{displaymath}

and in other forms,

\begin{displaymath}
\sigma(3^4) = 1 + 3\sigma(3^{4 - 1}) = \sigma(3^{4 - 1}) + 3^4 = \frac{3^{4+1}-1}{3-1} = 121.
\end{displaymath}

as partial sum of a geometric series. $\lhd$

Therefore, using (\ref{M}), our main equation is

\begin{displaymath}
\prod_{i=1}^{m}  (\sum_{j=0}^{a_i}p_i^{j})=
2\prod_{i=1}^{m} p_i^{a_i}
\end{displaymath}

or, equivalently,

\begin{equation}
\prod_{i=1}^{m} \sigma(p_i^{a_i})=
2\prod_{i=1}^{m} p_i^{a_i}
\mbox{\quad (} p_i \not =2, a_i > 0, m \geq 1 \mbox{).}
\label{E1}
\end{equation}

To solve the problem we have to exhibit a literal odd number with this property, or prove that it's impossible.

If $\sigma(n) < 2n$, we say $n$ is {\bf deficient}.
Every number with $m=1$, that is every prime power, is deficient.  Every {\it odd} number with two distinct primes is also known to be deficient (See \cite{r2}, Nocco's Theorem). If $\sigma(n) > 2n$, we say $n$ is {\bf abundant}.

For $m \geq 3$, looking back at \ref{E1}, we 
note that only one $\sigma(p_i^{a_i})$ must be even, 
and of the form $2S$, with odd $S$. The corresponding
$p_i$ is usually called the "special prime". So we can give it a special index, like, $i=s$ ($p_s$ is the special prime).
Moreover, to make it possible to balance the
equation \ref{E1}, with the above conditions, $a_s$ must be odd. Further 
investigation reveals \cite{r1} that, whatever $a_s$ might be, $p_s + 1$ divides $\sigma(p_s^{a_s})$;
hence $\frac{p_s + 1}{2}$ divides $n$. Put another way, there is at least one $p_i < p_s$.

Furthermore, it should be obvious that every $a_i$, other than $a_s$ must be even.

Finally, the last result we need from \cite{r1} is that, if $n$ is perfect, every divisor $d$ of $n$ would
be deficient.

\section{Odd perfect numbers don't exist}

\begin{thh}
If a natural number is odd, then it's not a perfect number.
\label{T1}
\end{thh}

\begin{proof}
We argue by contradiction. Let $n > 1$ be an odd perfect number and assume, on the contrary to the statement above, that $n$ is perfect. Hence, \ref{E1} holds.

We arrange the primes in ascending order,

\begin{equation}
p_1 < p_2 < \cdots < p_m
\end{equation}

relabeling the primes if necessary. Of course $p_s$ (the special prime) apears in this chain of inequalities. We
know that $p_s \not = p_1$, but we make no assumptions about its exact place.

We also have the list of divisors of $n$ in descending order,

\begin{equation}
n > p_1^{a_1 - 1} \prod_{i=2}^m p_i^{a_i} > \cdots > p_1p_2 > \cdots > p_1 > 1.
\end{equation}

We define,

\begin{equation}
\lambda = p_1^{a_1 - 1} \prod_{i=2}^m p_i^{a_i} = \frac{n}{p_1},
\end{equation}

and we note that (A) $\lambda$ is the greatest [proper] divisor of n, and (B) $a_1 - 1 \not = 0$, because $a_1$ is even.

Finally, we divide $n - 1$ by $\lambda$, and keep dividing the remainder by the next greatest divisor of $n$. The
process will be described in detail.

{\bf Note.} 
We wish to mention that the reader is already familiar with this idea. (In
calculating the greates common divisor of two numbers, say, or converting a number from base 10 to base $b$, etc.
The general idea is clear, but the difference is in several details, like "keep dividing the remainder" or "keep dividing the quotient" etc.) $\lhd$

Now, we write:

\begin{equation}
n - 1 = x \lambda + r  \quad (0 \leq r < \lambda),
\end{equation}

and solve for $x$. Of course, $n = p_1 \lambda = (p_1 - 1) \lambda + \lambda$, implies that,

\begin{equation}
n - 1 = (p_1 - 1) \lambda + (\lambda - 1),  \quad  (x = p_1 - 1) \ \ (r = \lambda - 1).
\end{equation}

(The division algorithm is carried out correctly.)

Next, we divide {\it the remainder} by $\frac{\lambda}{p_1} = \lambda_2$ which must be the next largest divisor of $n$. We obtain,

\begin{equation}
n - 1 = (p_1 - 1) \lambda + (p_1 - 1) \lambda_2 + (\lambda_2 - 1).
\end{equation}

After several more steps, $a_1$-steps to be precise, 
the remainder will be $\frac{n}{p_1^{a_1}} - 1$, that is, the $p_1^{a_1}$ factor is vanished; and the
way we arranged the primes, from smallest to largest, we are sure that the next greatest divisor of $n$ is 
$p_2^{a_2-1}\prod_{i=3}^m p_i^{a_i}$, which goes into the last remainder $p_2 - 1$ times. This
procedure goes on, until it reaches to the last two steps: $p_m^2 - 1 = (p_m - 1)p_m + (p_m - 1)$ and finally,
$p_m - 1 = 1 \cdot (p_m - 1) + 0$. It takes $\sum_{i=1}^m a_i$ steps to finish. 
(We'll demonstrate this with a numerical example, shortly.)
The result looks like,

\begin{equation}
n - 1 = (p_1 - 1) \lambda + (p_1 - 1) \lambda_2 + \cdots + (p_i - 1)\lambda_j + \cdots + (p_m - 1)p_m + (p_m - 1).
\end{equation}

{\bf Note.} 
What we've accomplished so far, is simply, converting the number $n - 1$ to the {\it mixed base} $p_1, p_2,
..., p_m$. We think of the quotients $p_1 - 1, p_2 - 1, \cdots, p_m - 1$ as  {\it digits}, calculated
in a {\bf top-down} process. (But it's
not necessary for the reader to adopt this point of view, as long as she agrees that the quotients are
unique, and therefore the representation of $n - 1$ as above, is well defined.)

Now, in the usual algorithm for converting a number, for instance 523, to any base, let's say base 10, we first
calculate the right-most digit of the result as $(523 \mod 10) = 3$. (This number is already written in base 10,
but we are investigating how it was obtained.) If we call this the {\bf bottom-up} process, we can define 
a top-down process, too, in which the  {\it left-most} (i.e. the  {\bf most significant}) digit is calculated first.
We start by observing that $10^2 < 523 < 10^3$ the we solve $523 = x 10^2 + r$, hence $x=5$.
In this case, $10^2$ is the greatest power of 10 which goes into 523.

In the above argument, we knew for sure that $\lambda$ was the greatest divisor of $n$, and we
knew in each step, what was exactly the next greatest divisor. In other words, in a concise manner,
we just converted $n - 1$ to a mixed base.
$\lhd$

{\bf Example.} Perform the above algorithm for $n = 2205$.

{\it solution.} We have $n = 3^2 \cdot 5 \cdot 7^2$, in that order of primes. Then,

\begin{displaymath}
2205 = 2 (3^1 \cdot 5 \cdot 7^2) + (3^1 \cdot 5 \cdot 7^2),
\end{displaymath}

\begin{displaymath}
2205 = 2 (3^1 \cdot 5 \cdot 7^2) + 2 (5 \cdot 7^2) + (5 \cdot 7^2),
\end{displaymath}

\begin{displaymath}
2205 = 2 (3^1 \cdot 5 \cdot 7^2) + 2 (5 \cdot 7^2) + 4(7^2) + 7^2,
\end{displaymath}

\begin{displaymath}
2205 = 2 (3^1 \cdot 5 \cdot 7^2) + 2 (5 \cdot 7^2) + 4(7^2) + 6(7) + 7.
\end{displaymath}

The result is:

\begin{displaymath}
2205 - 1 = 2 (3^1 \cdot 5 \cdot 7^2) + 2 (5 \cdot 7^2) + 4(7^2) + 6(7) + 6.
\end{displaymath}

Note that by adding $1$ to both sides of this equation, all steps will run in {\it reversed} order.
We explain that, if we think of $2,2,4,6,6$ as digits, as having a {\it carry}, and all
the digits that are carried to the next position above, will add up to the number $2205$,
which plays a role similar to $10^n$ in base 10. By analogy, if desired,
think of adding $1$ to $9999$ in base 10.

Again, the proof avoids any explicit use of these concepts, for simplicity.
$\lhd$

We expand the sum of divisors of $n$, as:

\begin{equation}
2n = \sum_{d|n} d
\end{equation}

then,

\begin{equation}
n - 1 = \lambda + \lambda_2 + \cdots + p_2p_1 + \cdots + p_1.
\end{equation}

We like to compare this form, with the result of the previous calculation for $n - 1$.

Given the uniqueness of the steps taken above, to obtain,

\begin{equation}
n - 1 = (p_1 - 1) \lambda + (p_1 - 1) \lambda_2 + \cdots + (p_i - 1)\lambda_j + \cdots + (p_m - 1)p_m + (p_m - 1),
\end{equation}

We realize,

\begin{equation}
(\sum_{\substack{d|n \\ d < p_m}} d)  - 1 = p_m - 1
\end{equation}

{\bf Note.} 
The sum runs over all divisors of $n$ (less than $p_m$), not all divisors of $p_m$.
$\lhd$

Otherwise, $\sum_{\substack{1 < d|n \\ d < p_m}} \not = p_m - 1$, and either the last remainder or
the next to the last remainder wouldn't be $p_m - 1$, which is a contradiction, if the above procedure is well-defined.

Similarly,

\begin{equation}
(\sum_{\substack{d|n \\ d < p_m^2}} d)  - 1 = p_m^2 - 1,
\label{U1}
\end{equation}

up to,

\begin{equation}
(\sum_{\substack{d|n \\ d < p_m^{a_m}}} d)  - 1 = p_m^{a_m} - 1.
\label{U2}
\end{equation}

We notice that if we add $1$ to both sides of $n - 1$ in

\begin{equation}
n - 1 = (p_1 - 1) \lambda + (p_1 - 1) \lambda_2 + \cdots + (p_i - 1)\lambda_j + \cdots + (p_m - 1)p_m + (p_m - 1),
\end{equation}

we get a cascading effect of sums, (a chain reaction, if you like) that runs backwards, and in each step we get a divisor of $n$,
from $p_m^1$ to $p_m^{a_m}$, then from $p_{(m-1)}^{1}p_m^{a_m}$ to $p_{(m-1)}^{a_{(m-1)}}p_m^{a_m}$, all the way back to $\lambda$. For each step, we can deduce an equality like \ref{U1} and \ref{U2}. The last of them being,

\begin{equation}
(\sum_{\substack{d|n \\ d < \lambda}} d)  - 1 = \lambda - 1.
\end{equation}

or,

\begin{equation}
\sum_{\substack{d|n \\ d < \lambda}} d = \lambda.
\end{equation}

which works just fine for the last reversed step (to be added to $(p_1 - 1)\lambda$ and get $n$), but comparing this result with,

\begin{equation}
n = \sum_{\substack{ d|n \\ d < n}} d
\end{equation}

and it's expanded form, we notice that we must account for $\lambda$ and everything else less than $\lambda$, in the
right-hand side.

{\bf Note.} 
We can't assume that every divisor $d$ of $n$ has the property $d | \lambda$. It holds for {\it some} divisors of $n$. (There are divisors with $p_1^{a_1} | d$, too.) But we know that for every divisor $d$ of $n$ we have $d < \lambda$,
except of course for $\lambda$ itself, and $n$.
$\lhd$

But that simply leads to,

\begin{equation}
n = \sum_{\substack{ d|n \\ d <n}} d = \lambda + \sum_{\substack{ d|n \\ d < \lambda}} d = 2\lambda
\end{equation}

Therefore $n$ is odd and divisible by 2.
From this contradiction we conclude that an odd number $n$ cannot be perfect. This completes the proof.
\end{proof}

(The proof could be represented as a direct argument, rather than an ad absurdum.)

\section{Conclusion}
We proved that perfect numbers are always even, and therefore, always related to Mersenne primes. Whether or not the set of Mersenne primes is infinite, is another interest\-ing open problem \cite{r3}.

\end{document}